\documentclass[11pt,oneside,
final,
a4paper]{amsart}
\usepackage{ifdraft}
\usepackage{verbatim}
\usepackage[notref,notcite]{showkeys}
\usepackage{amssymb,latexsym}
\usepackage{mathtools}
\usepackage[hmargin=3cm,vmargin=3.5cm]{geometry}
\usepackage[T1]{fontenc}
\usepackage{ae}
\usepackage[utf8]{inputenc}
\usepackage[dvips]{graphicx}
\usepackage{subfigure}
\usepackage{nicefrac}
\theoremstyle{plain}
\newtheorem{theorem}{\textbf{Theorem}}[section]
\newtheorem{lemma}[theorem]{\textbf{Lemma}}
\newtheorem{corollary}[theorem]{\textbf{Corollary}}
\theoremstyle{definition}
\newtheorem{definition}[theorem]{\textbf{Definition}}
\newtheorem{example}[theorem]{\textbf{Example}}
\theoremstyle{remark}
\newtheorem{remark}[theorem]{\textbf{Remark}}
\theoremstyle{remark}
\newtheorem{remarks}[theorem]{\textbf{Remarks}}
\numberwithin{equation}{section}
\numberwithin{subsection}{section}
\newcommand{\dimh}{\operatorname{dim_H}}
\newcommand{\dimp}{\operatorname{dim_P}}

\newcommand{\ydimb}{\operatorname{\overline{dim}_B}}

\newcommand{\ydiml}{\operatorname{\overline{dim}_{loc}}}
\newcommand{\adiml}{\operatorname{\underline{dim}_{loc}}}

\newcommand{\diml}{\operatorname{dim_{loc}}}

\newcommand{\R}{\mathbb{R}}

\newcommand{\N}{\mathbb{N}}

\newcommand{\Q}{\mathbf{Q}}
\renewcommand{\a}{\textit{\textbf{a}}}
\newcommand{\m}{\textit{\textbf{m}}}

\newcommand{\I}{\mathcal{I}}

\newcommand{\Po}{\mathcal{P}}
\newcommand{\IN}{I^{\N}}
\newcommand{\ii}{\textnormal{\texttt{i}}}
\newcommand{\jj}{\textnormal{\texttt{j}}}

\newcommand{\avlim}[1]{\lim_{#1\to\infty}\frac{1}{#1}}
\newcommand{\iin}[1]{\ii|_{#1}}

\title{Local dimensions of measures on infinitely generated
self-affine sets}
\author{Eino Rossi}
\address{Eino Rossi\\
Department of Mathematics and Statistics\\
P.O. Box 35 (MaD) FI-40014 University of Jyväskylä\\
Finland}
\email{eino.rossi@jyu.fi}
\thanks{The author would like to thank his supervisor Antti K\"aenm\"aki for
his help and guidance and the University of Jyv\"askyl\"a for financial
support}
\date{\today}

\begin{document}
\begin{abstract}
We show the existence of the local dimension of an invariant
probability measure on an infinitely generated self-affine set, for almost all
translations. This implies that an ergodic probability measure is exactly
dimensional. Furthermore the local dimension equals the minimum of the local
Lyapunov
dimension and the dimension of the space.
\end{abstract}
\maketitle
\section{Introduction}
The \textit{upper} and \textit{lower local dimensions} of a locally finite Borel
measure $\mu$, denoted by $\ydiml(\mu,x)$ and $\adiml(\mu,x)$
respectively, are the $\limsup$ and $\liminf$ of the ratio
\[
\frac{\log\mu\left( B(x,r)\right)}{\log r},
\]
as $r\to 0$. When they agree, we say that the \textit{local dimension}, denoted
by $\diml(\mu,x)$, exists and equals to this common value. If the local
dimension is constant almost everywhere, we say that $\mu$ is \textit{exactly
dimensional}. The local dimension does not only give information about the
geometry of the measure, but also about the support of the measure. For example,
if the upper local dimension of $\mu$ is smaller than $t$ for all $x\in A$, then
the packing dimension of $A$ is at most $t$, see e.g. \cite[Proposition
2.3(d)]{Falconer3}.
%

Our main interest is to study the local dimensions of the canonical projection
$\pi\mu$ of an invariant Borel probability measure $\mu$ onto a self-affine set.
In 2009, Feng and Hu \cite{Feng} showed that the local dimension of $\pi\mu$
exists almost everywhere if
the underlying iterated function system, IFS for short, is conformal. They also
showed that the local dimension exists if the mappings of the IFS satisfy
$f_i(x)=A_i x+a_i$ and the
matrices $A_i$ commute. When $\mu$ is ergodic, these results give that $\mu$ is
exactly dimensional. The general affine
case remained open. In 2011, Falconer and Miao \cite{FalcMiao}
calculated the
local dimension in a specific affine case. They showed that $\pi\mu$ is exactly
dimensional for Lebesgue almost all translation vectors $a\in\R^{d\kappa}$,
where $\kappa$ is the number of mappings in the IFS,
assuming that $\mu$ is a Bernoulli measure and that
$\sup_{i}||A_i||<\frac{1}{2}$, see \cite[Theorem 6.1]{FalcMiao}. By taking a
minor change in the proof of \cite[Theorem 4.3]{JPS} we can have the same
result for any ergodic measure. This was noted in a very restrictive case by
Barral and Feng in \cite[Theorem 2.6]{FengBarralv1} and giving the general proof
was one of our motivations at the beging of this work. In \cite{BarralFeng2013},
the published version of \cite{FengBarralv1}, it is mentioned that this
generalization is also known by the authors of \cite{JPS}. However the proof is
not written out. All the works mentioned here assume that the IFS is
finitely generated.

Our main result, Theorem \ref{mainthm}, generalizes the results mentioned above.
We show that even in the infinitely generated case, the local dimension of
an invariant Borel measure exists, assuming again that
$\sup_{i}||A_i||<\frac{1}{2}$.
As a corrollary we get that ergodic measures are exactly dimensional. We also
remark how to obtain estimates for the local dimensions that hold for
all translations, with only assuming that the mappings $A_i$ are
contractive. Finally, we make
remarks on the connections of our results to the dimensions of the limit set.

Let us now introduce some notation. Let $I$ be a finite or countable set. We
define
$I^*=\bigcup_{n=1}^{\infty}I^{n}$.
If $I$ is finite, we say that $\IN$ is \textit{finitely generated} and
otherwise $\IN$ is \textit{infinitely generated}. When $\ii\in I^*$, we denote
by
$\ii\jj$ the symbol obtained by juxtaposing $\ii$ and $\jj$. Furthermore, for
$\ii\in I^*$, we set $[\ii]=\{\ii\jj\colon\jj\in\IN\}$ and call this set a
\textit{cylinder} of $\ii$. When $\ii=(i_1,i_2,\ldots)$ we denote
$\iin{n}=(i_1,\ldots,i_{n})$. On the symbol space $\IN$ we consider the
\textit{left shift} $\sigma$, defined by
$\sigma(i_1,i_2,i_3,\ldots)=(i_2,i_3,\ldots)$ and study Borel
measures that are \textit{invariant} with respect to this shift, that is
$\mu(B)=\mu(\sigma^{-1}B)$ for all borel sets $B$. An invariant measure is
called \textit{ergodic},
if for all Borel sets $B$ with $B=\sigma^{-1}B$, we have $\mu(B)=0$ or
$\mu(B)=1$. We denote
the set of invariant and ergodic Borel
probability measures on $\IN$ by
$\mathcal{M}_{\sigma}(\IN)$ and $\mathcal{E}_{\sigma}(\IN)$ respectively.
Throughout the paper, $\mu$ denotes a Borel probability measure.
By $\pi\mu$, we mean the push-forward measure $\mu\circ \pi^{-1}$.

For each $i\in I$, we fix an invertible $d\times d$ matrix $A_i$ and a
translation vector $a_i\in\Q$, where $\Q=[-\frac{1}{2},\frac{1}{2}]^d$. Due to
Kolmogorov extension theorem $\Q^{\N}$ supports a natural probability measure
$\m=(\mathcal{L}^d|_{\Q})^{\N}$.
We assume that $\sup_{i\in I}||A_i||=\overline{\alpha}<1$ and consider the IFS
$\{f_i\}_{i\in I}$,
where
$f_i(x)=A_i x+a_i$, and the \textit{canonical projection} $\pi_{\a}:\IN\to\R^d$
defined by
$\{\pi_{\a}(\ii)\}=\bigcap_{n\in\N}f_{\iin{n}}(B(0,R))$,
where $f_{\iin{n}}=f_{i_1}\circ\dots\circ f_{i_n}$ and $R$ is so large that
$f_i(B(0,R))\subset B(0,R)$ for all $i\in I$. We
call $F_{\a}=\bigcup_{\ii\in\IN}\pi_{\a}(\ii)$ the \textit{limit set} of this
IFS. It is
not restrictive to assume that each $a_i$ is in the cube $\Q$, since this
is just a matter of scaling the limit set. This only exculudes
the case where $\sup_i|a_i|=\infty$.

The \textit{singular values}
$||A_{\iin{n}}||=\alpha_1(\iin{n})\geq\dots\geq\alpha_d(\iin{n})>0$ of
$A_{\iin{n}}=A_{i_1}\cdots A_{i_n}$ are the lengths of the principal semiaxis of
the ellipsoid
$A_{\iin{n}}(B(0,1))$. For $0\leq s<d$, the \textit{singular value function} is
defined as
\begin{equation*}
\phi^s(\iin{n})=
\alpha_{1}(\iin{n})\cdots\alpha_{k}(\iin{n})\alpha_{k+1}(\iin{n})^{s-k},
\end{equation*}
where $k$ is the integer part of $s$. When $s\geq d$, we set
$\phi^s(\iin{n})=\left(\alpha_{1}(\iin{n})\cdots\alpha_{d}(\iin{n})\right)^{
\nicefrac{s}{d
}}$.

For convenience, fix partitions
$\Po_n=\{[\ii]\}_{\ii\in I^n}$, and set
$H_{\mu}(\Po_n)=-\sum_{\ii\in I^n}\mu[\ii]\log\mu[\ii]$. 
\textit{Enropy} and \textit{energy} of $\mu\in\mathcal{M}_{\sigma}(\IN)$,
defined by
\[
h_{\mu}=-\avlim{n}H(\Po_n)\qquad\text{and}\qquad
\Lambda_{\mu}(s)=\avlim{n}\int_{\IN} \log\phi^s(\iin{n})\: d\mu
\]
respectively, are the basic tools in the study of ergodic measures in the field
of
iterated function systems. In order to work with invariant measures,
we need to localize these concepts. By Theorems \cite[Theorem 7 in section
2]{Parry} and \cite[Theorem 10.1]{Walters}, for
$\mu\in\mathcal{M}_{\sigma}(\IN)$, there exist $L^1(\mu)$ functions
$h_{\mu}(\ii)$ and
$\Lambda_{\mu}(s,\ii)$ so that
\begin{equation}
\label{SMsuberg}
h_{\mu}(\ii)=-\avlim{n} \log\mu[\iin{n}]\qquad\text{and}\qquad
\Lambda_{\mu}(s,\ii)=\avlim{n} \log\phi^s(\iin{n}),
\end{equation}
for $\mu$ almost all $\ii\in\IN$ and
\begin{equation}
\label{SMsubergI}
\int_{\IN} h_{\mu}(\ii)\: d\mu=h_{\mu}\qquad\text{and}\qquad
\int_{\IN} \Lambda_{\mu}(s,\ii)\: d\mu=\Lambda_{\mu}(s).
\end{equation}
Furthermore, for $\mu\in\mathcal{E}_{\sigma}(\IN)$, we
have $h_{\mu}(\ii)=h_{\mu}$ and $\Lambda_{\mu}(s,\ii)=\Lambda_{\mu}(s)$
for $\mu$ almost all $\ii\in\IN$. We call $h_{\mu}(\ii)$ the \textit{local
entropy} of $\mu$ at $\ii$ and $\Lambda_{\mu}(s,\ii)$ the \textit{local energy}
of $\mu$ at $\ii$. In order to use \cite[Theorem 7 in section 2]{Parry}, we
need to assume that $H(\Po_n)<\infty$ at some level $n$.
Since
$h_{\mu}=\avlim{n}  H(\Po_n)=\inf_{n\in\N}\frac{1}{n} H(\Po_n)$ we can
equivalently assume
that $h_{\mu}<\infty$.

We define the \textit{measure-theoretical pressure} function
$P_{\mu}(\:\cdot\:,\ii):[0,\infty]\to\R$ by
\[
P_{\mu}(s,\ii)=-\avlim{n}\log\frac{\mu[\iin{n}]}{\phi^s(\iin{n})},
\]
when $\ii$ is so that both equations in
\eqref{SMsuberg} hold. If $h_{\mu}<\infty$ the limit exists for
$\mu$ almost all $\ii\in\IN$. When
$h_{\mu}(\ii)<\infty$ or $\Lambda_{\mu}(s,\ii)>-\infty$, then $P_{\mu}(s,\ii)$
is just
$h_{\mu}(\ii)+\Lambda_{\mu}(s,\ii)$.\ifdraft{\textbf{ Selvää katsomallakin.}}{}

It is not yet said, that there exists $\ii\in\IN$, so that
$\avlim{n}\log\phi^s(\iin{n})$ exists for all $s$. Fortunately, this happens
for $\mu$
almost all $\ii\in\IN$. By repetitive use of the second equation in
\eqref{SMsuberg}, we get that for $\mu$ almost all $\ii\in\IN$,
the limit $\avlim{n}\log\alpha_l(\iin{n})$ exists for all $1\leq l \leq d$. We
call these values the \textit{Lyapunov exponents} of $\mu$ at $\ii$ and denote
them by $\lambda_l(\mu,\ii)$.
For $s<d$, it now
easily follows that
\begin{equation}
\label{energy}
\Lambda_{\mu}(s,\ii)=\lambda_1(\mu,\ii)+\dots+\lambda_k(\mu,\ii)+(s-k)\lambda_{
k+1 } (\mu,\ii) ,
\end{equation}
where $k$ is the integer part of $s$, with the interpretation that
$0\cdot(-\infty)=0$. If $s\geq d$, we get
$\Lambda_{\mu}(s,\ii)=\frac{s}{d}(\lambda_1(\mu,\ii)+\dots+\lambda_d(\mu,\ii))$.
From this
we see, that
$\Lambda_{\mu}(\:\cdot\:,\ii)$ is strictly decreasing function with
$\Lambda_{\mu}(0,\ii)=0$. Also we see that
$\Lambda_{\mu}(\:\cdot\:,\ii)$ has at most one point of discontinuity and at
this point it is
continuous from left. The point of discontinuity equals to
$\min\{k:\lambda_{k+1}(\mu,\ii)=-\infty\}$.

With the assumption $h_{\mu}<\infty$, we have that for $\mu$ almost all
$\ii\in\IN$, the equations in \eqref{SMsuberg} hold for all $s$. Also, the first
equation in \eqref{SMsubergI} gives
that $h_{\mu}(\ii)<\infty$ for $\mu$ almost all $\ii\in\IN$. In this light, we
give the following definition.
\begin{definition}
\label{ldimension}
Let $\mu\in\mathcal{M}_{\sigma}(\IN)$ and $h_{\mu}<\infty$. When $\ii$ is so
that $h_{\mu}(\ii)<\infty$ and both equations in \eqref{SMsuberg} hold, the
\textit{local Lyapunov dimension}
of $\mu$ at
$\ii$, denoted by
$\dim_{LY}(\mu,\ii)$, is defined to be the
infimum of the numbers $s$, for which $P_{\mu}(s,\ii)<0$.
\end{definition}
We remark that, for ergodic $\mu$, the above functions
$h_{\mu}(\ii),\lambda_l(\mu,\ii),\Lambda_{\mu}(s,\ii),P_{\mu}(s,\ii)$ and
$\dim_{LY}(\mu,\ii)$ are constants for $\mu$ almost all $\ii$. In such case,
we use the notations $h_{\mu},\lambda_l(\mu),\Lambda_{\mu}(s),P_{\mu}(s)$ and
$\dim_{LY}(\mu)$ to emphasize the independence of $\ii$. We are now ready to
state our main result.
\begin{theorem}
\label{mainthm}
Assume that $\mu\in \mathcal{M}_{\sigma}(\IN)$, $h_{\mu}<\infty$,
$\sup_{i\in I}||A_i||<\frac{1}{2}$ and that there exists $s\in[0,\infty)$
so that $0>P_{\mu}(s,\ii)>-\infty$. Then
$\diml(\pi_{\a}\mu,\pi_{\a}(\ii))=\min\{d,\dim_{LY}(\mu,\ii)\}$
for $\mu$ almost all $\ii\in\IN$ and $\m$ almost all $\a\in\Q^{\N}$.
\end{theorem}
We  only need the assumption $0>P_{\mu}(s,\ii)>-\infty$ in the proof of the
upper bound to ensure that $\lambda_{k+1}(\mu,\ii)>-\infty$, where $k$ is the
integer part of $\dim_{LY}(\mu,\ii)$.

\section{Local dimensions of invariant measures}
In this section we prove Theorem \ref{mainthm}. The proof is divided into upper
and
lower estimates, namely to Theorems \ref{lowerestimate} and
\ref{upperestimate}. We remark that Theorem \ref{lowerestimate} was proven in
\cite[Proposition 4.4]{JPS} for an ergodic measure on a finitely generated
affine IFS.
\begin{theorem}
\label{lowerestimate}
Assume that $\mu\in \mathcal{M}_{\sigma}(\IN)$, $h_{\mu}(\ii)<\infty$ and
$\sup_{i\in I}\alpha_1(i)<\frac{1}{2}$.
Then we have that
$\adiml(\pi_{\a}\mu,\pi_{\a}(\ii))\geq\min\{d,\dim_{LY}(\mu,\ii)\}$ for $\mu$
almost all $\ii\in\IN$ and $\m$ almost all $\a\in\Q^{\N}$.
\end{theorem}
\begin{proof}
Assume first that $\dim_{LY}(\mu,\ii)\leq d$.
For arbitrary $\varepsilon>0$,
we choose $\gamma(\ii)=\dim_{LY}(\mu,\ii)-2\varepsilon$ and
$\theta(\ii)=\dim_{LY}(\mu,\ii)-\varepsilon$ Since $P_{\mu}(\:\cdot\:,\ii)$ is
strictly decreasing, we
find $\varepsilon'>0$, so that $\Lambda_{\mu}(\theta(\ii),\ii)\geq
-h_{\mu}(\ii)+2\varepsilon'$. By Egoroff's
theorem, for each $\delta>0$ there is a measurable set $H_{\delta}\subset\IN$
and integer
$N_{\delta}$, such that $\mu(\IN\setminus H_{\delta})<\delta$ and
\[
\frac{1}{n}\log \mu[\iin{n}]\leq -h_{\mu}(\ii)+\varepsilon'\leq
\Lambda_{\mu}(\theta(\ii),\ii)-\varepsilon'\leq
\frac{1}{n}\log\phi^{\theta(\ii)}(\iin{n})
\]
for all $n\geq N_{\delta}$ and $\ii\in H_{\delta}$. Therefore we find a constant
$c'>0$, independent of $\ii$,
so that
\begin{equation}
\label{estimate1}
\mu[\iin{n}]\leq c' \phi^{\theta(\ii)}(\iin{n})
\end{equation}
for all $n\in\N$ and $\ii\in H_{\delta}$. Next we consider the integral
\[
\int_{\Q^{\N}}\frac{d\m(\a)}{|\pi_{\a}(\ii)-\pi_{\a}(\jj)|^{\gamma(\ii)}}
=\int_{\Q^{\N}}\int_{\Q}\frac{d\mathcal{L}^{d}(a_1)}{|\pi_{\a}(\ii)-\pi_{\a}
(\jj)|^ { \gamma(\ii) } }d\m(\a'),
\]
where $\a=(a_1,\a')\in\Q^{\N}$. We can make the change of variable in the inner
integral as in \cite[Lemma 3.1]{Falconer88}. By using this Lemma with Fubini's
theorem, and then inequality \eqref{estimate1} and the
properties of the singular value function, we get
\begin{align*}
\int_{\Q^{\N}}\int_{H_{\delta}}\int_{\IN}\frac{d\mu(\jj)d\mu(\ii)d\m(\a)}{|\pi_{
\a}
(\ii)-\pi_{\a}
(\jj)|^{\gamma(\ii)}}
&\leq c \int_{H_{\delta}}\int_{\IN}
(\phi^{\gamma(\ii)}(\ii\wedge\jj))^{-1}d\mu(\jj)d\mu(\ii)\\
&\leq c \int_{H_{\delta}}
\sum_{n=1}^{\infty}(\phi^{\gamma(\ii)}(\iin{n}))^{-1}\mu[\iin{n}] d\mu(\ii)\\
&\leq cc' \int_{H_{\delta}}
\sum_{n=1}^{\infty}(\phi^{\gamma(\ii)}(\iin{n}))^{-1}\phi^{\theta(\ii)}(\iin{n})
d\mu(\ii)\\
&\leq cc' \int_{H_{\delta}}
\sum_{n=1}^{\infty}2^{n(\gamma(\ii)-\theta(\ii))}
d\mu(\ii)\\
&\leq cc' \sum_{n=0}^{\infty} 2^{-n\varepsilon}\int_{H_{\delta}}d\mu(\ii)
<\infty,
\end{align*}
where $\ii\wedge\jj=\iin{\min\{k-1:i_k\neq j_k\}}$ and $c$ is the constant from
\cite[Lemma 3.1]{Falconer88}, independent of
$\ii$ and $\jj$. Originally, the bound
of the
norms of the linear maps in \cite[Lemma 3.1]{Falconer88} is $\frac{1}{3}$, but
by
\cite[Proposition 3.1]{Solo}, $\frac{1}{2}$ suffices. Now we have that
\begin{equation}
\label{almosta}
\int_{H_{\delta}}\int_{\IN}\frac{d\mu(\jj)d\mu(\ii)}{|\pi_{\a}(\ii)-\pi_{\a}
(\jj)|^{\gamma(\ii)
}}<\infty
\end{equation}
for $\m$ almost all $\a\in\Q^{\N}$. Next we fix $\a$ so that \eqref{almosta}
holds. We deduce that the integral
$\int_{\IN}|\pi_{\a}(\ii)-\pi_{\a}(\jj)|^{-\gamma(\ii)}\:d\mu(\jj)$ is finite
for $\mu$ almost
all $\ii\in H_{\delta}$ and so we find constants $M(\ii)$ for $\mu$ almost
every $\ii\in
H_{\delta}$, so that
$\int_{\IN}|\pi_{\a}(\ii)-\pi_{\a}(\jj)|^{-\gamma(\ii)}\:d\mu(\jj)<M(\ii)$.
This implies that $\pi_{\a}\mu(B(\pi_{\a}(\ii),r))\leq r^{\gamma(\ii)}M(\ii)$
for all $r>0$ and for $\mu$ almost
all $\ii\in H_{\delta}$.

We have obtained that $\adiml(\pi_{\a}\mu,\pi_{\a}(\ii))\geq \gamma(\ii)$ for
$\mu$ almost all
$\ii\in H_{\delta}$ and $\m$ almost all $\a\in \Q^{\N}$. Since
$\delta$ was arbitrary, this also holds for $\mu$ almost all $\ii\in \IN$.

If $\dim_{LY}(\mu,\ii)>d$, then we get the proof by choosing
$\theta(\ii)=d$ and $\gamma(\ii)=d-\varepsilon$.
\end{proof}

\begin{theorem}
\label{upperestimate}
If $\mu\in \mathcal{M}_{\sigma}(\IN)$, $h_{\mu}<\infty$ and
$\Lambda_{\mu}(s,\ii)>-\infty$ for some $s>\dim_{LY}(\mu,\ii)$,
then $\ydiml(\pi_{\a}\mu,\pi_{\a}(\ii))\leq\min\{d,\dim_{LY}(\mu,\ii)\}$ for
$\mu$
almost all $\ii\in \IN$ and for all $\a\in\Q^{\N}$.
\end{theorem}
\begin{proof}
As mentioned in the introduction, we follow the lines of the proof of \cite[Theorem 4.3]{JPS}.

We may assume that $\dim_{LY}(\mu,\ii)< d$. Fix an integer $k$,
so that
$k\leq\dim_{LY}(\mu,\ii)<k+1$. We have that $\pi_{_a}[\iin{n}]\in
f_{\iin{n}}(B(0,R))$ for some $R\in\N$. The ellipsoid $f_{\iin{n}}(B(0,R))$ can
be covered
by a rectangular box, call it $B(\iin{n})$, with side-lengths
$2R\alpha_1(\iin{n}),\ldots,2R\alpha_d(\iin{n})$. We can cover
$B(\iin{n})$ with $N(\iin{n})$ non-overlapping ``half-open'' boxes with
side-lengths
\[
\alpha_{k+1}(\iin{n}),\ldots,\alpha_{k+1}(\iin{n}),\alpha_{k+2}(\iin{n}),\ldots,
\alpha_{d}
(\iin{n}),
\]
where
$N(\iin{n})\leq(2R)^d\alpha_{1}(\iin{n})\cdots\alpha_{k}(\iin{n})\alpha_{k+1}
(\iin{n})^{-k}$.
Let $P_n(\ii)$ be the box that contains $\pi_{\a}(\ii)$, and let
$Q_n(\ii):=[\iin{n}]\cap \pi_{\a}^{-1}(P_n(\ii))$. In other words, $Q_n(\ii)$ is
the part of the cylinder $[\iin{n}]$ that gets
projected into
$P_n(\ii)$. For fixed $j$ we define
\[
A_n^j:=\{\ii\in\IN:\mu(Q_n(\ii))\geq
2^{-\nicefrac{n}{j}}\frac{\mu[\iin{n}]}{N(\iin{n})}\}
\]
for all $n\in\N$. Now we have
\[
\mu(\IN\setminus A^j_n)
=\mu(\bigcup_{\ii\in\IN}Q_n(\ii)\setminus A^j_n)
=\sum_{Q_n(\ii)\not\subset A^j_n}\mu(Q_n(\ii))
\leq\sum_{\ii\in I^n} N(\ii)2^{-\nicefrac{n}{j}}\frac{\mu[\ii]}{N(\ii)}
= 2^{-\nicefrac{n}{j}}.
\]
Thus for the set $A^j:=\bigcup_{N\in\N}\bigcap_{n=N}^{\infty}A^j_n$ we have
\[
\mu(A^j)=\lim_{N\to\infty}\mu(\bigcap_{n=N}^{\infty} A^j_n)
=1-\lim_{N\to\infty}\mu\left(\bigcup_{n=N}^{\infty} (\IN\setminus A^j_n) \right)
\geq 1-\lim_{N\to\infty}\sum_{n=N}^{\infty}2^{-\nicefrac{n}{j}}
=1.
\]
By definition, for all $\ii\in A^j$, we find $M(\ii)\in\N$ such that the
inequality
\begin{equation}
\label{monta}
\mu(Q_n(\ii))\geq 2^{-\nicefrac{n}{j}}\frac{\mu[\iin{n}]}{N(\iin{n})}
\end{equation}
holds for all $n\geq M(\ii)$.

We assumed that $\Lambda_{\mu}(s,\ii)>-\infty$ for some $s>\dim_{LY}(\mu,\ii)$.
This implies that
$\lambda_{l}(\mu)>-\infty$ for all $1\leq l\leq k+1$. For
$\gamma(\ii)>\dim_{LY}(\mu,\ii)$, we have
\begin{align}
\label{limit}
\lim_{n\to\infty}
\frac{\log \mu[\iin{n}]-\log N(\iin{n})}{\log\alpha_{k+1}(\iin{n-1})}
&=
\lim_{n\to\infty}
\frac{\frac{1}{n}\left( \log\mu[\iin{n}]
-\sum_{l=1}^{k}\log\alpha_l(\iin{n})-\log\alpha_{k+1}(\iin{n})^{-k}\right)}
{\frac{n-1}{n}\frac{1}{n-1}\log\alpha_{k+1}(\iin{n-1})} \nonumber\\
&\leq
\frac{h_{\mu}(\ii)+\lambda_{1}(\mu,\ii)+\dots+\lambda_{k}(\mu,\ii)}{-\lambda_{
k+1} (\mu,\ii) }+k <\gamma(\ii) 
\end{align}
for $\mu$ almost all $\ii\in\IN$. The first inequality follows by the
definition of $N(\iin{n})$ and the fact that $h_{\mu}(\ii)$ and
$\lambda_{l}(\mu,\ii)$ are finite for $1\leq l\leq k+1$ and
$0>\lambda_{k+1}(\mu,\ii)$. The second
inequality follows by \eqref{energy}, since $P_{\mu}(\gamma(\ii),\ii)<0$. In
the calculation, we have omitted the constant $(2R)^d$ from $N(\iin{n})$, since
it has no effect
on the result.

Let $r_l$ be any sequence of positive numbers converging to zero. For each $l$,
we find an integer $n_l$, so that $\sqrt{d}\alpha_{k+1}(\iin{n_l})\leq
r_l<\sqrt{d}\alpha_{k+1}(\iin{n_l-1})$. To avoid complicated notation, we only
write $n$ instead of $n_l$. We have $\ii\in Q_n(\ii)$ and
$\pi_{\a}Q_n(\ii)\subset P_n(\ii)$ and the greatest side-length of $P_n(\ii)$
is $\alpha_{k+1}(\iin{n})$.
Therefore
we have $\pi_{\a}Q_n(\ii)\subset
B(\pi_{\a}(\ii),\sqrt{d}\alpha_{k+1}(\iin{n}))$.
Using this and \eqref{monta} and \eqref{limit}, we get that
\begin{align*}
\limsup_{l\to\infty}\frac{\log \pi_{\a}\mu
B(\pi_{\a}(\ii),r_l)}{\log r_l}
&\leq\limsup_{n\to\infty}\frac{\log \pi_{\a}\mu
B(\pi_{\a}(\ii),\sqrt{d}\alpha_{k+1}(\iin{n}))}{\log\sqrt{d}\alpha_{k+1}(\iin{
n-1}
) }\\
&\leq\limsup_{n\to\infty} \frac{\log \mu Q_n(\ii)}{\log \alpha_{k+1}(\iin{n-1})}
\\
&\leq \limsup_{n\to\infty}\left( \frac{\log 2^{-\nicefrac{n}{j}}}{\log
\alpha_{k+1}(\iin{n-1})}
+\frac{\log \mu[\iin{n}]-\log N(\iin{n})}{\log \alpha_{k+1}(\iin{n-1})}
\right)\\
&\leq j^{-1}\frac{\log 2}{-\log \overline{\alpha}}+\gamma(\ii),
\end{align*}
where $\overline{\alpha}=\sup_{i\in I}\alpha_1(i)<1$. Since $j$ and the
sequence  $r_l$ were arbitrary and
$\mu(A^j)=1$ for all $j\in \N$, we have obtained
$\ydiml(\pi_{\a}\mu,\pi_{\a}(\ii))\leq \gamma(\ii)$ for $\mu$ almost all
$\ii\in\IN$.
\end{proof}

\begin{remark}
It is natural to ask, what can be said of the local dimensions, when one only
assumes $\sup_{i\in I}\alpha_1(i)\leq\overline{\alpha}<1$, and what results can
be obtained for all translations $\a$. Observe that Theorem
\ref{upperestimate} already applies to this case. By using an essentially identical
proof as the proof of \cite[Theorem 2.6]{Feng}, one can get the following estimates.

Assume that $\mu\in\mathcal{M}_{\sigma}(\IN)$, $h_{\mu}<\infty$ and $\log
\alpha_d(\iin{1})\in L^1(\mu)$. Then we have for $\mu$
almost all $\ii\in\IN$ and for all $\a\in\Q^{\N}$ that
\begin{equation}
\label{fengtheorem}
\frac{h_{\mu}^{\pi}(\ii)}{-E_{\mu}(\log \alpha_d(\iin{1})|\I)}
\leq \adiml(\pi\mu,\pi(\ii))
\leq \ydiml(\pi\mu,\pi(\ii))
\leq\frac{h_{\mu}^{\pi}(\ii)}{-E_{\mu}(\log \alpha_1(\iin{1})|\I)},
\end{equation}
where
$h_{\mu}^{\pi}(\ii)$ is the \emph{local projection entropy}
defined as in \cite[Definition 2.1]{Feng},
$m$ is so that $H(\Po_m)<\infty$, and $\I$ is the $\sigma$-algebra of
$\sigma$ invariant sets. For the definition of the conditional
expectation $E_\mu$, see \cite{Parry}.
If the index set $I$ is finite then \eqref{fengtheorem} is strictly included in \cite[Theorem 2.6]{Feng}.

The assumptions $h_{\mu}<\infty$ and $\log\alpha_d(\iin{1})\in L^1(\mu)$ 
are needed in the ergodic theorems that are used in
the proof and the number $m$ can be chosen to be the least integer for which
$H(\Po_m)<\infty$. In the finitely generated case these assumptions are
of course satisfied and $m=1$. We also note that the proof of
\cite[Theorem 6.2]{Feng}, which a more general version of \cite[Theorem 2.6]{Feng}, deals with a direct product of two IFS and the conditional measures used there are not needed to obtain \eqref{fengtheorem}.

In most cases, the upper bound in equation \eqref{fengtheorem} is not as good
as the
result of Theorem \ref{upperestimate}. However, in the exceptional case, where
Theorem \ref{lowerestimate} does not hold, the upper estimate in
\eqref{fengtheorem} might give
a better estimate since $h_{\mu}^{\pi}\leq h_{\mu}$, see
\cite[Proposition 4.1]{Feng}.
\end{remark}

\section{Pressure function and dimensions of the limit set}
In order to determine the Hausdorff dimension of the limit set $F_{\a}$,
one often considers the \textit{pressure} function defined by
\[
P(s)=\avlim{n}\log\sum_{\ii\in I^n} \phi^s(\ii).
\]
In the finitely generated
setting it is known that if $\max_{i\in I}||A_i||<\frac{1}{2}$, then the
Hausdorff dimension of $F_{\a}$ equals to the zero of the pressure for
$\mathcal{L}^{d\kappa}$ almost all $\a\in\R^{d\kappa}$, where $\kappa$ denotes
cardinality of the index set $I$, see \cite{Falconer88}. In
\cite[Theorem B]{KReeve},
Käenmäki and Reeve generalize this result for an infinitely
generated affine IFS, with the extra assumption of quasi-multiplicativity, see
\cite[(2.1)]{KReeve} for the definition. Since
their results on the Hausdorff dimension of the limit set are closely related
to our results on measures, we give some notes on this paper.

The pressure function satisfies $P(s)\geq P_{\mu}(s)$ for all
$\mu\in\mathcal{M}_{\sigma}(\IN)$ and all $s\in[0,\infty)$, see \cite[Lemma
2.2]{KReeve}. Furthermore, if the singular value function is
quasi-multiplicative and $s>s_{\infty}=\inf\{s:P(s)<\infty\}$, then there exists
an
ergodic measure $\mu_s$, called the Gibbs measure, satisfying
$P(s)=P_{\mu_s}(s)$, $h_{\mu_s}<\infty$ and $\Lambda_{\mu_s}(s)>-\infty$,
\cite[Theorems 3.5 and 3.6 and Lemma
4.2]{KReeve}. In example \ref{exsinfty} we show that $P(s)$ can be nonzero
everywhere. In this
case, any ergodic measure $\mu$ with $h_{\mu}<\infty$ satisfies
$\dim_{LY}(\mu)<s_{\infty}$. This
follows since $P(s_{\infty})\geq
P_{\mu}(s_{\infty})$ and $P_{\mu}$ is continuous from left. The next
theorem gives a necessary and sufficient condition for the existence of the zero
of the pressure function under the quasi-multiplicativity assumption.
\begin{lemma}
\label{pressure}
Suppose that the singular value function $\phi^s(\ii)$ is
quasi-multiplicative for all $0\leq s\leq d$. Then $P(s)$ is continuous and
strictly decreasing on the interval $[s_{\infty},\infty)$. Furthermore if
$P(s_{\infty})\geq 0$, then there exists a unique $s$ satisfying $P(s)=0$.
\end{lemma}
Käenmäki and Vilppolainen have proved a similar result, \cite[Lemma
2.1]{KVsubaffine}, and we will make use of that proof. Their lemma deals with a
finitely generated IFS, but some parts of the proof apply directly to the
infinitely generated case.
\begin{proof}[Proof of Lemma \ref{pressure}]
It is easy to see, that $P(s)$ is decreasing and thus it is finite for all
$s>s_{\infty}$. As in \cite[Lemma 2.1]{KVsubaffine}, we deduce
that for any $s>s_{\infty}$, we have
\[
P(s)-P(s+\delta)\geq -\delta\log \sup_{i\in I}\alpha_1(\ii),
\]
which gives that $P(s)$ is strictly decreasing for $s>s_{\infty}$ and that
$\lim_{s\to\infty} P(s)=-\infty$. Now we only neeed to show the continuity. By
inspecting the proof of \cite[Lemma 2.1]{KVsubaffine}, we
get that
$P(s)$ is convex on intervals $[m,m+1]$. Since $P(s)$ is also decreasing,
we get that $P(s)$ is left-continuous for all $s>s_{\infty}$. Since
$\phi^s(\ii)$ is quasi-multiplicative, $P(s)$ can be approximated pointwise by
continuous functions from
below, namely by the pressures of finite sub-systems, see \cite[Proposition
3.2]{KReeve}. Again, using the
fact that
$P(s)$ is decreasing, we get right-continuity. Especially, $P(s)$ is
right-continuous at $s_{\infty}$. Note also that quasi-multiplicativity was only
used to get the right-continuity.
\end{proof}
The lower local dimension of the Gibbs measure is also estimated in
\cite[Theorem 4.1]{KReeve}.
By Lemma \ref{pressure} and Theorem \ref{mainthm} we get the
following corollary. Due to Lemma \ref{pressure} the assumption of the existence of
$s_0$ in \cite[Theorem 4.1]{KReeve} can be relaxed to $P(s_\infty)\geq 0$.
\begin{corollary}
\label{Gibbs}
Suppose that the singular value function $\phi^s(\ii)$ is
quasi-multiplicative for all $0\leq s\leq d$ and $P(s_{\infty})>0$. Then there
exists $s_0$ so that $P(s_0)=0$ and the Gibbs measure $\mu_{s_0}$ satisfies
$\Lambda_{\mu_{s_0}}(s_0)>-\infty$. If in addition
$\Lambda_{\mu_{s_0}}(s_0+\delta)>-\infty$ for some $\delta>0$, then
\[
\diml(\pi_{\a}\mu_{s_0},\pi_{\a}(\ii))=s_0,
\]
for $\mu$ almost all $\ii\in\IN$ and $m$ almost all $\a\in\Q^{\N}$.
\end{corollary}
\begin{proof}
The existence of $s_0$ is clear by Lemma \ref{pressure}. The assumption
$\Lambda_{\mu_{s_0}}(s_0+\delta)>-\infty$ is only needed when $s$ is an
integer, to ensure that we
may use Theorem \ref{mainthm}. The
result follows since
$s_0=\dim_{LY}(\mu,\ii)$.
\end{proof}

By \cite[Theorem B]{KReeve}, $\dimh F_{\a}=\sup\{\dimh
\pi_{\a}(J^{\N}):J\subset I \text{ is finite}\}$ for $\m$ almost all $\a$. We
do not know whether a similar approximation holds for $\dimp$ and
$\ydimb$. Recalling \cite[Theorem 10.1]{Falconer3}, one
could use Corollary \ref{Gibbs} and hope for
results on packing dimension of the limit set. The problem is that we only
know the local dimension of $\mu_{s_0}$ for almost all $\ii$ and not for
all
$\ii$. Mauldin and Urbanski have given an
example of an infinitely generated self similar set $F$ satisfying
the open set condition, for which $\dimh F<\dimp F$, see \cite[Example
5.2]{MauldinUrbanski}. On the other hand, for all finite subsystems it holds
that $\dimh \pi_{\a}(J^{\N})=\dimp \pi_{\a}(J^{\N})$, see \cite{Falconer89}.
Therefore the dimension approximation property does not hold for this, or
similar examples. Note also that $\ydimb F_{\a}=\dimp F_{\a}$ for infinitely
generated self-affine sets $F_{\a}$ by
\cite[Theorem 3.1]{MauldinUrbanski}. The following theorem gives an estimate for
the relation between Hausdorff and
packing dimensions of infinitely generated self affine sets. For $x\in F_{\a}$,
we set the notation $L_n(x)=\{f_{\ii}(x):\ii\in I^n\}$.
\begin{theorem}
Let $\{f_i\}_{i\in I}$, be an infinitely generated affine IFS. Then we have that
\[
\sup_{\stackrel{x\in F_{\a}}{n\in\N}}\{\dimh F_{\a},\ydimb L_n(x) \}\leq
\dimp F_{\a} \leq\sup_{\stackrel{x\in F_{\a}}{n\in\N}}\{s_0,\ydimb L_n(x)\},
\]
where $s_0=\inf\{s:\avlim{n}\log\sum_{\ii\in I^n}\alpha_1(\ii)^s=0\}$.
\end{theorem}
\begin{proof}
We have $\dimp F_{\a}=\ydimb F_{\a}$ by \cite[Theorem 3.1]{MauldinUrbanski} and so the
first inequality is trivial. The proof of the last inequality is essentially the same
as the proof of \cite[Lemma 2.8]{MauldinUrbanski2},
since $||f_{\ii}'||=\alpha_1(\ii)$.
\end{proof}
Note that if $s_0\leq 1$, then $s_0=\inf\{s:P(s)<0\}=\dimh F_{\a}=\dimp F_{\a}$ for
$\m$ almost all $\a\in\Q^{\N}$ by \cite[Theorem B]{KReeve}.
\section{Examples and final remarks}
Here we give some examples on the
entropies and pressures of measures. In example
\ref{exmpressure} we show that the measure-theoretical pressure function can be
non-zero everywhere and in example \ref{exsinfty} we show that the
pressure function can be non-zero everywhere, as mentioned earlier. In the
examples, we make use of \textit{Bernoulli
measures}: Fix reals $0\leq p_i\leq 1$ so
that $\sum_{i=1}^{\infty}p_i=1$. The unique measure
satisfying $\mu[\iin{n}]=p_{i_1}p_{i_2}\cdots p_{i_n}$ is called a
Bernoulli measure. It is well known that Bernoulli measures are ergodic.
It is also easy to see that the entropy of a Bernoulli measure can be infinite.

\begin{example}
\label{exmpressure}
($P_{\mu}(s)\neq 0$ everywhere)
Let
$\mu$ be a Bernoulli measure with $\mu[i]=c(i+1)^{-2}$, where
$c=(\frac{\pi^2}{6}-1)^{-1}$. Let
\[
A_i=\begin{bmatrix*} 2\mu[i] &0\\0& c4^{-i} \end{bmatrix*}.
\]
We can now calculate
\[
h_{\mu}=-\sum_{i=2}^{\infty}ci^{-2}\log ci^{-2}=
\log c+2c\sum_{i=2}^{\infty}i^{-2}\log i<\infty
\]
and thus $\mu$ is a probability measure with finite entropy.
Also, by induction we see that
$
\lambda_1(\mu)=\sum_{i=1}^{\infty}\mu[i]\log2\mu[i]=\log 2-h_{\mu}
$
and
\begin{align*}
\lambda_2(\mu)
&=c\sum_{i=1}^{\infty}(i+1)^{-2}\log c4^{-i}=-\infty
\end{align*}
Thus $\mu$ is a Bernoulli measure with finite entropy and $P_{\mu}(t)\geq\log2$
for all $t\leq 1$ and $P_{\mu}(t)=-\infty$ for all $t>1$. Since $\sup_{\ii\in I}
\alpha_1(i)=2\mu[1]=\frac{2c}{4}<\frac{1}{4}$, Theorem \ref{lowerestimate} gives
that
$\adiml(\pi_{\a}\mu,\pi_{\a}(\ii))\geq 1$ for $\mu$ almost all $\ii\in\IN$ and
$m$
almost all $\a\in\Q^{\N}$.
\end{example}
\begin{example}
\label{exsinfty}
($P(s)\neq 0$ everywhere)
Let
$c_i=i^{-\frac{1}{2}}$,
$d_i=i^{-1}$ and $A_i=\begin{bmatrix}c_i & 0\\ 0 & d_i\end{bmatrix}$
for all $i\in\N$. Now
$A_{\ii}=\begin{bmatrix}c_{\ii} & 0\\ 0 & d_{\ii}\end{bmatrix}$ for all $\ii\in
I^n$, where $c_{\ii}=c_{i_1}\cdots c_{i_n}$ and $d_{\ii}=d_{i_1}\cdots d_{i_n}$.
Therefore, for all $t=1+s$, we have
\[
\phi^t(\ii)=\frac{1}{i_1^{\frac{1}{2}}}\cdots\frac{1}{i_n^{\frac{1}{2}}}
\cdot\frac{1}{i_1^s} \cdots\frac {1}{i_n^s},
\]
which implies
\[
\sum_{\ii\in I^n}\phi^t(\ii)=\left(\sum_{\ii\in\N}\frac{1}{i^{\frac{1}{2}+s}}
\right)^n\qquad\text{ and }\qquad
P(t)=\log\sum_{\ii\in\N}\frac{1}{i^{\frac{1}{2}+s}}.
\]
Choose $I=\{ \lfloor i(\log i)^2\rfloor:i\geq n_0\}$. Now we have that
\[
P(\frac{3}{2})=\log \sum_{\ii\in
I}\frac{1}{i^{\frac{1}{2}+\frac{1}{2}}}=\log\sum_{i=n_0}^{\infty}\frac{1}{
\lfloor
i(\log i)^2\rfloor}<0
\]
for $n_0$ large enough. For all $t<\frac{3}{2}$ we get $P(t)=\infty$, since
$\log i \leq i^{\delta}$ for large $i$ when $\delta>0$.
\end{example}

We end with final remarks on the assumptions and results of this paper.
\begin{remarks}
\label{final}
\begin{enumerate}
\item
Considering the proof of Theorem \ref{upperestimate}, suppose that
$\lambda_{k+1}(\mu,\ii)=-\infty$. We face difficulties at \eqref{limit} since we
are to calculate the limit
\[
\lim_{n\to\infty}\frac{\log
\alpha_{k+1}(\iin{n})}{\log\alpha_{k+1}(\iin{n-1})}.
\]
Since $\alpha_{k+1}(i)\to 0$ as $i\to\infty$, there are sequences for which the
above
limit is infinite. If one has extra information about the support of the
measure then the set of these sequences can be studied. For example, if
$\mu\left(\IN\setminus\bigcup J^{\N}\right)=0$,
where the union is over all finite sets $J\subset I$, then we find
constants $c(\ii)$ for almost all $\ii$ so that $\alpha_{k+1}(\iin{n})\geq
c(\ii)\alpha_{k+1}(\iin{n-1})$ and the set of the exceptional sequences is of
measure zero. Unfortunately these measures are rather trivial. This can be seen
from \cite[Lemma 2.3]{KVmoran}. Note that one can always use the sequence
$\alpha_{k+1}(\iin{n})$
to obtain the estimate $\adiml(\pi_{\a}\mu,\pi_{\a}(\ii))\leq
\dim_{LY}(\mu,\ii)$.
\item
Since we assumed that the limit set $F$ is bounded it is reasonable to also
assume that
$\alpha_d(i)\to 0$ as $i\to\infty$. Therefore we could have
$-\sum_{i\in
I}\mu[i]\log \alpha_d(\iin{1})=\infty$ and so the assumption
$\log \alpha_d(\iin{1})\in L^1$ in equation \eqref{fengtheorem} is necessary.
\item
Considering the finitely generated case, suppose that $\#I=\kappa$ and that
$s_0$ is the zero of the
pressure function.
Käenmäki proved the existence of an equilibrium measure 
$\mu_{s_0}$ in \cite{Kmeasure}. For this measure, $\dim_{LY}(\mu_{s_0})$ equals
to
$s_0$. By Theorem \ref{mainthm}, we get that $\dimh(F)\geq s_0$ for Lebesque
almost all $\a\in\R^{d\kappa}$. This shows that
we can not remove
the assumption $\sup_{i\in I}||A_i||<\frac{1}{2}$ from Theorem
\ref{mainthm}. For examples where $\dimh(F)< s_0$, see
\cite{Edgarexample,PU,SimoSolo}. Also
there are
examples showing that for particular $\a$, Theorem \ref{mainthm} can not hold,
see e.g. \cite[Example 9.11]{Falconer2}. The size of the set of these
exceptional translation has been studied by Falconer and Miao in
\cite{FalcMiaoExceptional}. 
\item
Supposing that $h_{\mu}^{\pi}(\ii)<\infty$, we may slightly modify the
definition of the Lyapunov dimension, namely by setting
\[
\dim_{LY}^{\pi}(\mu,\ii)=\inf\{s:h_{\mu}^{\pi}(\ii)-\Lambda_{\mu}(s,\ii)<0\}.
\]
Perhaps we could have
$\diml(\pi_{\a}\mu,\pi_{\a}(\ii))=\min\{d,\dim_{LY}^{\pi}(\mu,\ii)\}$ for $\mu$
almost
all $\ii\in\IN$ and all $\a\in\Q^{\N}$, when $\mu\in\mathcal{E}_{\sigma}(\IN)$,
$h_{\mu}^{\pi}<\infty$ and
$\sup_{i\in I}||A_i||<1$.
\end{enumerate}
\end{remarks}

\bibliographystyle{plain}
\bibliography{luettelo.bib}

\begin{thebibliography}{10}

\bibitem{BarralFeng2013}
Julien Barral and De-Jun Feng.
\newblock {Multifractal formalism for almost all self-affine measures}.
\newblock {\em Comm. Math. Phys.}, 318(2):473--504, 2013.

\bibitem{Edgarexample}
G.~A. Edgar.
\newblock {Fractal dimension of self-affine sets: some examples}.
\newblock {\em Rend. Circ. Mat. Palermo (2) Suppl.}, (28):341--358, 1992.
\newblock Measure theory (Oberwolfach, 1990).

\bibitem{Falconer88}
K.~J. Falconer.
\newblock {The {H}ausdorff dimension of self-affine fractals}.
\newblock {\em Math. Proc. Cambridge Philos. Soc.}, 103(2):339--350, 1988.

\bibitem{Falconer89}
K.~J. Falconer.
\newblock {Dimensions and measures of quasi self-similar sets}.
\newblock {\em Proc. Amer. Math. Soc.}, 106(2):543--554, 1989.

\bibitem{Falconer2}
Kenneth Falconer.
\newblock {\em {Fractal geometry}}.
\newblock John Wiley \& Sons Ltd., Chichester, 1990.
\newblock Mathematical foundations and applications.

\bibitem{Falconer3}
Kenneth Falconer.
\newblock {\em {Techniques in fractal geometry}}.
\newblock John Wiley \& Sons Ltd., Chichester, 1997.

\bibitem{FalcMiaoExceptional}
Kenneth Falconer and Jun Miao.
\newblock {Exceptional sets for self-affine fractals}.
\newblock {\em Math. Proc. Cambridge Philos. Soc.}, 145(3):669--684, 2008.

\bibitem{FalcMiao}
Kenneth Falconer and Jun~Jie Miao.
\newblock {Local dimensions of measures on self-affine sets}.
\newblock arXiv:1105.2411v1, 2011.

\bibitem{FengBarralv1}
De-Jun Feng and Julien Barral.
\newblock {Multifractal formalism for almost all self-affine measures}.
\newblock arXiv:1110.6578v1, 2011.

\bibitem{Feng}
De-Jun Feng and Huyi Hu.
\newblock {Dimension theory of iterated function systems}.
\newblock {\em Comm. Pure Appl. Math.}, 62(11):1435--1500, 2009.

\bibitem{JPS}
Thomas Jordan, Mark Pollicott, and K{\'a}roly Simon.
\newblock {Hausdorff dimension for randomly perturbed self affine attractors}.
\newblock {\em Comm. Math. Phys.}, 270(2):519--544, 2007.

\bibitem{Kmeasure}
Antti K{\"a}enm{\"a}ki.
\newblock {On natural invariant measures on generalised iterated function
  systems}.
\newblock {\em Ann. Acad. Sci. Fenn. Math.}, 29(2):419--458, 2004.

\bibitem{KReeve}
Antti K{\"a}enm{\"a}ki and Henry~WJ Reeve.
\newblock {Multifractal analysis of Birkhoff averages for typical infinitely
  generated self-affine sets}.
\newblock {\em EMS: Journal of Fractal Geometry}, to appear.

\bibitem{KVmoran}
Antti K{\"a}enm{\"a}ki and Markku Vilppolainen.
\newblock {Separation conditions on controlled {M}oran constructions}.
\newblock {\em Fund. Math.}, 200(1):69--100, 2008.

\bibitem{KVsubaffine}
Antti K{\"a}enm{\"a}ki and Markku Vilppolainen.
\newblock {Dimension and measures on sub-self-affine sets}.
\newblock {\em Monatsh. Math.}, 161(3):271--293, 2010.

\bibitem{MauldinUrbanski}
R.~Daniel Mauldin and Mariusz Urba{\'n}ski.
\newblock {Dimensions and measures in infinite iterated function systems}.
\newblock {\em Proc. London Math. Soc. (3)}, 73(1):105--154, 1996.

\bibitem{MauldinUrbanski2}
R.~Daniel Mauldin and Mariusz Urba{\'n}ski.
\newblock {Conformal iterated function systems with applications to the
  geometry of continued fractions}.
\newblock {\em Trans. Amer. Math. Soc.}, 351(12):4995--5025, 1999.

\bibitem{Parry}
William Parry.
\newblock {\em {Topics in ergodic theory}}, volume~75 of {\em {Cambridge Tracts
  in Mathematics}}.
\newblock Cambridge University Press, Cambridge, 1981.

\bibitem{PU}
F.~Przytycki and M.~Urba{\'n}ski.
\newblock {On the {H}ausdorff dimension of some fractal sets}.
\newblock {\em Studia Math.}, 93(2):155--186, 1989.

\bibitem{SimoSolo}
K{\'a}roly Simon and Boris Solomyak.
\newblock {On the dimension of self-similar sets}.
\newblock {\em Fractals}, 10(1):59--65, 2002.

\bibitem{Solo}
Boris Solomyak.
\newblock {Measure and dimension for some fractal families}.
\newblock {\em Math. Proc. Cambridge Philos. Soc.}, 124(3):531--546, 1998.

\bibitem{Walters}
Peter Walters.
\newblock {\em {An introduction to ergodic theory}}, volume~79 of {\em
  {Graduate Texts in Mathematics}}.
\newblock Springer-Verlag, New York, 1982.

\end{thebibliography}
\end{document}